\documentclass[11pt, letter]{article}

\usepackage[english]{babel}
\usepackage[utf8]{inputenc}
\usepackage[T1]{fontenc}
\usepackage{stmaryrd}
\usepackage{footmisc}

\usepackage{soul}

\usepackage{fullpage}
\usepackage{mathrsfs}
\usepackage{amsmath}
\usepackage{bbm}
\usepackage{graphicx}
\usepackage{stackengine}
\usepackage{amsthm}
\usepackage{subcaption}
\usepackage{amssymb}

\usepackage{thmtools}
\usepackage{thm-restate}

\usepackage[colorinlistoftodos]{todonotes}
\usepackage[colorlinks=true, allcolors=blue]{hyperref}
\usepackage{cleveref}

\newcommand{\ml}[1]{{\color{blue}{{#1}}}}

\title{Recognizing Leaf Powers and Pairwise Compatibility Graphs is NP-Complete}

\date{}

\author{Max Dupr\'e la Tour\thanks{McGill University: \texttt{maxduprelatour@gmail.com}} \and 
Manuel Lafond\thanks{Universit\'e de Sherbrooke: \texttt{manuel.lafond@usherbrooke.ca}} \and 
Ndiam\'e Ndiaye\thanks{McGill University. \texttt{ndiame.ndiaye@mail.mcgill.ca}} }
\usepackage{amsfonts}
\newtheorem{theorem}{Theorem}[section]

\newtheorem{lemma}[theorem]{Lemma}
\newtheorem{lm}[theorem]{Lemma}

\newtheorem{cor}[theorem]{Corollary}
\newtheorem{proposition}[theorem]{Proposition}

\newtheorem{claim}[theorem]{Claim}
\theoremstyle{remark}
\newtheorem*{remark}{Remark}

\DeclareMathOperator{\diam}{diam}

\newcommand{\set}[1]{\left\{#1\right\}}

\DeclareMathOperator*{\argmin}{arg\,min}

\newcommand{\bb}[0]{\mathbb}

\newcommand{\N}{\bb{N}}

\begin{document}

\maketitle

\begin{abstract}
Leaf powers and pairwise compatibility graphs were introduced over twenty years ago as simplified graph models for phylogenetic trees. Despite significant research, several properties of these graph classes remain poorly understood. In this paper, we establish that the recognition problem for both classes is NP-complete. We extend this hardness result to a broader hierarchy of graph classes, including pairwise compatibility graphs and their generalizations, multi-interval pairwise compatibility graphs. 
\end{abstract}

\section{Introduction}

One of the fundamental questions in graph theory is whether a given graph can be represented or described using a simpler or more structured metric space. Graphs derived from metric spaces have been extensively studied. For instance, \emph{unit distance graphs}, usually studied in Euclidean spaces of small dimension, are formed by connecting vertices that are exactly one unit apart. Similarly, \emph{unit disk graphs} connect vertices whose distance is at most one. Expanding on this idea, \emph{annulus graphs} connect vertices whose distances lie within a specific interval $[\theta_1,\theta_2]$.

In this paper, we focus on tree metrics, where the vertex set consists of the leaves of an edge-weighted tree.
Unit disk graphs defined on tree metrics correspond to \emph{leaf powers} (LP), a class of graphs whose vertices are the leaves of some tree and whose edges are pairs of leaves with a distance below a certain threshold in the tree.
The tree is called a \emph{leaf root} and, if $k$ is the threshold in question, the graph is a $k$-leaf power.
Similarly, annulus graphs on tree metrics correspond exactly to \emph{pairwise compatibility graphs} (PCG), and instead require the edges to be pairs of leaves whose distance is between two given thresholds. Leaf powers and pairwise compatibility graphs were introduced over two decades ago \cite{NISHIMURA200269,Kearney2003} as simplified models for phylogenetic trees, aiming to capture the evolutionary relationships between species. 

Outside of computational biology, leaf powers also have algorithmic applications.  
Notably, they form a proper subclass of graphs of \emph{mim-width} $1$, as shown by Jaffke et al. in~\cite{jaffke2019mimIII}\footnote{Graphs of mim-width $1$ admit a branch decomposition in which every bipartition yields a bipartite graph whose maximum induced matching of $1$.}.  Several NP-hard problems become polynomial time solvable on this graph class~\cite{jaffke2020mimI,jaffke2020mimII}, but the complexity of recognizing graphs of mim-width $1$ is open.  It could be NP-hard given the recently established the para-NP-hardness of the mim-width parameter~\cite{bergougnoux2025mim}.  Recognizing leaf powers efficiently would circumvent this by still allowing the use of known algorithms for graphs of bounded mim-width on them.  
Fellows et al.~\cite{fellows2008leaf} also identified algorithmic applications for $k$-leaf powers. They show that deciding whether a $k$-leaf power $G$ satisfies an M1O graph expression, a restriction of monadic second order logic, can be done in linear time if $k$ is fixed and the leaf root is known (it is in fact fixed-parameter tractable in $k$).  
 Moreover, $k$-leaf powers have bounded clique-width and a similar meta-theorem of Courcelle is also applicable~\cite{courcelle2000linear}, although the complexity of the algorithm of Fellows et al. is actually practical.
Their result also generalizes to PCGs whose thresholds are fixed.

Despite their significance, these classes remain poorly understood, and several graph theoretic problems concerning these leaf powers and PCGs remain open. The purpose of this research is to resolve one such long-standing open problem: the computational complexity of their recognition.

This question was partially answered by Lafond~\cite{Lafond2023}, who showed that $k$-leaf powers could be recognized in polynomial time for any fixed $k$.  However, fixing $k$ significantly limits both practical and theoretical application of leaf powers, and the techniques in~\cite{Lafond2023} do not appear to extend to the broader class of leaf powers, nor to PCGs and generalizations.  For leaf powers, Nevries and Rosenke conjectured that leaf powers could be characterized as the strongly chordal graphs that forbid one of seven induced subgraphs, which would immediately lead to a polynomial time algorithm~\cite{nevries2016towards}.  Lafond disproved this in~\cite{lafond2017strongly} by constructing an infinite family of non-strongly chordal obstructions.  The author also showed that deciding whether a chordal graph contains one of these obstructions is NP-complete.  This does not imply the hardness of recognizing leaf powers, and it would seem that other ideas are needed to solve this two decades old problem.

In this work, we finally establish that the recognition problems for both leaf powers and pairwise compatibility graphs are NP-complete. Furthermore, we extend this result to a broader hierarchy of graph classes, which we term \emph{generalized leaf powers}. A generalized leaf power of order $ q $ is a graph whose vertex set consists of the leaves of a positively weighted tree, called the \emph{generalized leaf root of order $q$} of that graph. There exist $ q $ distinct thresholds $ \theta_1, \dots, \theta_q $ such that two vertices are connected in the graph if and only if the distance between them in the tree is less than an odd number of these thresholds. This hierarchy corresponds precisely to \emph{multi-threshold} graphs defined on tree metrics rather than star metrics, where a star metric requires the vertex set to consist of the leaves of an edge-weighted star. Multi-threshold graphs were introduced in \cite{JamisonSpragueMultiT} as an extension of unit disk and annulus graphs on star metrics. For even $q$, $q$-generalized leaf powers correspond to the multi-interval PCG introduced in \cite{MultiPCG}. Note that $1$-generalized leaf powers are exactly leaf powers, and $2$-generalized leaf powers are exactly PCG. Our main result is the following theorem:

\begin{restatable}{thm}{main}\label{thm:main}
    For any integer $q\geq 1$, the $q$-generalized leaf powers recognition problem is NP-complete.
\end{restatable}

\begin{cor}
    The following problems are NP-complete:
    \begin{itemize}
        \item Recognition of leaf powers.
        \item Recognition of pairwise compatibility graphs.
        \item Recognition of multi-interval PCGs with $k$ intervals for any constant $k \geq 2$.
    \end{itemize}
\end{cor}

For leaf powers, we even obtain a slightly better result.  Recall that encoding an integer $k$ in unary means that it is written as a sequence of $k$ times the symbol $1$ (which makes $k$ require $k$ bits in the input, as opposed to $O(\log k)$ bits if it is encoded in binary).

\begin{restatable}{theorem}{unary}
The unary-$ k $-leaf power recognition problem, where we are given a graph $ G $ and a threshold $ k $ encoded in unary and must determine whether $ G $ is a $ k $-leaf power, is NP-complete.
\end{restatable}

In order to prove that the $q$-generalized leaf powers recognition problem is in NP for all $q$, we prove that if a graph with $n$ vertices is a $q$-generalized leaf power, there exists a corresponding generalized leaf root of order $q$ such that all the weights of the tree are integers, and bounded by $(2n)^n$ . For $q=1$ (leaf powers), the smallest threshold $\theta$ such that a graph can be defined as the leaf power of a integer weighted tree using $\theta$ as the threshold is called the \emph{leaf rank} of the graph. It directly follows that the leaf rank of a graph with $n$ vertices is bounded by $n \cdot (2n)^n$.

\begin{restatable}{thm}{lrank}
    A leaf power with $n$ vertices has a leaf rank less than $n \cdot (2n)^n$.
\end{restatable}

\subsection{Overview of the Reduction}

We prove that the \emph{triangle ordinal clustering} (TOC) problem, shown to be NP-complete in~\cite{SF06}, can be reduced to the recognition of leaf powers.  A chain of reductions then implies hardness for their generalizations.
In the TOC problem, we are given a partial order $\prec_S$ on the pairs of elements of a set $S$.  
Specifically, $\prec_S$ specifies triangle orders, meaning that for each $i, j, k \in S$, it gives a total order on the triple $\{i, j\}, \{i, k\}, \{j, k\}$ (for instance, $\{i, j\} \prec_S \{i, k\} \prec_S \{j, k\}$).  The goal is to determine whether there is a tree whose vertex set contains $S$, and such that the pairwise distances in the tree agree with the partial order. 

Let us first emphasize that, to our knowledge, this is the first use of the TOC problem in an NP-hardness reduction.  There are several other types of graph-to-tree embedding problems whose complexity remains open, and it is possible that the techniques developed here could  accelerate progress on some other difficult complexity questions.  One notable example is the recognition of mim-width $1$ graphs mentioned above, and several other similar problems are stated in~\cite{hogemo2025mapping}.

The TOC problem differs significantly from leaf power recognition, as its distance constraints are not expressed in terms of global ``close'' versus ``far'' distances.  Rather, in TOC, each $v \in S$ provides its own local view of elements that are closer or farther to itself.  That is, if we consider two other elements $x, y$, then $v$ knows which of $x$ or $y$ should be closer, but each triple allows a different threshold of closeness.  
Nonetheless, from $(S, \prec_S)$ we are able to construct a graph $G$, an instance of leaf power recognition, with gadgets that enforce 
satisfying these local $\prec_S$ rankings in a leaf root.  For each $v \in S$, $G$ contains a corresponding set of vertices $U_v$ used as a proxy for the desired ranking of $v$.  Given a tree $T$ satisfying our partial order $\prec_S$, we are able to construct a leaf root $T'$ of $G$ containing $T$ 
such that each $U_v$ induces a weighted star that orders leaves as in the ranking desired by $v$.  Conversely, if a leaf root of $G$ exists, it must essentially organize these vertices in the same manner, with the desired ranking of $v$, and can be converted to a tree that satisfies $\prec_S$.

The exact machinery required to achieve this is difficult to describe succinctly here, but among the new tools developed,  we derive and use several consequences of the well-known \emph{four point condition}—a property characterizing tree metrics—which may have applications to other problems.  
In our case, we use it to guarantee that the local $U_v$ stars described in $T'$ above are consistent with $v$ from the original tree $T$, i.e., when transforming $T$ into $T'$, 
we can guarantee that if $v$ is closer to a vertex $x$ than to a vertex $y$, then all vertices from a set $U_v$ are closer to $x$ than to $y$.

\subsection{Related work}

\paragraph{Bounds on the leaf rank:} A recent paper \cite{LowerBoundLR} describes a family of graphs that are leaf powers but have an exponential leaf rank. Specifically, the leaf rank of this family is in $\Omega(2^{n/4})$. 
The conclusion of~\cite{LowerBoundLR} and~\cite[Chapter 11]{hogemo2025mapping} also mention an upper bound of $2^{n^c}$ on the leaf rank of any leaf power, for some unspecified constant $c$ based on the complexity of solving linear programs.  Our upper bound of $n (2n)^n \in 2^{O(n \log n)}$ brings us closer to the lower bound, although there is still a gap.

\paragraph{Algorithmic results:}  Prior to Lafond's algorithm for $k$-leaf powers, polynomial-time algorithms were known for every $k \leq 6$~\cite{brandstadt2006structure,brandstadt2008structure,chang20073,ducoffe20194}.  
For $k \leq 4$, the algorithms follow from a characterization of $k$-leaf power in terms of chordality and a finite set of forbidden induced subgraphs, but this algorithmic avenue was shown impossible for $k \geq 5$~\cite{dupre2024k}.
For leaf powers without a fixed $k$, it was shown in~\cite{brandstadt2010rooted} that leaf powers admitting a caterpillar leaf root 
(without edge weights) could be recognized in polynomial time, as they coincide with unit interval graphs.  This result was extended in~\cite{benjamin2022recognition}, where  a polynomial time algorithm is given to recognize leaf powers that admit an edge-weighted caterpillar leaf root.  The authors also have a polynomial time algorithm for the class of leaf powers admitting a NeS
model where the main embedding tree is a star (see definitions therein).  
In a different direction, the authors of~\cite{le2023computing} ask for the leaf rank of a given leaf power, and show that this can be computed in linear time for trivially perfect graphs.  

As for PCGs and other generalizations of leaf powers, the only positive result we are aware of is a polynomial-time algorithm for PCGs whose corresponding tree is an edge-weighted star~\cite{xiao2020characterizing,kobayashi2022linear}.  Also worth pointing out, the NP-completenes of finding leaf power obstructions in chordal graphs mentioned above builds upon another ``tree conflict'' problem, called Maximum Quartet Compatibility~\cite{steel1992complexity}, which has similarities with TOC.  Although this problem has apparently failed to lead to a hardness proof for leaf powers, it is also a good candidate for future reductions.

Let us also mention that leaf powers are known to be sandwiched between rooted directed path graphs and strongly chordal graphs~\cite{brandstadt2010rooted}, which are two polynomial-time recognizable classes.  As for PCGs, almost no interaction with graph classes is known except for restricted cases, for example PCGs of star graphs~\cite{xiao2020characterizing}.

\paragraph{Max-Generalized-$PCG$-Recognition:} To our knowledge, the only established hardness result on leaf powers and their generalizations is on a specific variant of the $PCG$ recognition problem. 
In the \emph{Max-Generalized-$PCG$-Recognition} problem, we are given an integer $q$ and two graphs $G_1 = (V, E_1)$ and $G_2 = (V, E_2)$ on the same set of vertices, such that $E_1 \subset E_2$. The question is whether there exists a $PCG$-graph $G = (V, E)$ with $E_1 \subset E \subset E_2$ and a tree using thresholds $\theta_1, \theta_2$, such that for at least $q$ of the non-edges, the distance in the tree between the two endpoints of the non-edge is greater than $\theta_2$.  The Max-Generalized-$PCG$-Recognition problem was introduced and proven to be NP-hard in \cite{durocher2015graphs}. This problem can be viewed as a restriction of the $PCG$-\emph{Sandwich} problem\footnote{For any graph class $\mathcal{C}$, the $\mathcal{C}$-Sandwich problem is defined as follows: Given two graphs $G_1 = (V, E_1)$ and $G_2 = (V, E_2)$ sharing the same set of vertices, with $E_1 \subset E_2$, is there a graph $G = (V, E)$ such that $E_1 \subset E \subset E_2$ and $G$ belongs to $\mathcal{C}$?}.
Even for some class of graphs $\mathcal{C}$ that can be recognized in polynomial time, the $\mathcal{C}$-Sandwich problem might still be NP-hard. An example is the class of chordal graphs, as demonstrated in \cite{GOLUMBIC1995449}.







\section{Generalizing Leaf Powers}
Let's begin by formally defining leaf powers. Let $G = (V, E)$ be a simple, finite graph. All graphs in this work are undirected, but for clarity we may write edges in $E$ as ordered pairs $(u, v)$, with the understanding that edges are undirected.
We say that $G$ is a \emph{leaf power} if there exists a tree $T$ and an integer threshold $\theta$ such that:

\begin{itemize} 
    \item The vertex set $V$ is the set of the leaves of $T$. 
    \item For any pair of distinct vertices $u, v \in V$, there is an edge $(u,v) \in E$ if and only if the distance between $u$ and $v$ in $T$, denoted as $d_T(u, v)$, is at most $\theta$. 
\end{itemize}

Here, $d_T$ represents the distance metric induced by the tree $T$, with two adjacent vertices having a distance of~$1$. Such a tree is called a \emph{leaf root} of $G$. For a leaf power $G$, the smallest possible threshold of a leaf root of $G$ is called the \emph{leaf rank} of $G$.

We now extend this definition to introduce a generalized version of leaf powers, which incorporates weighted edges and multiple thresholds. A graph $G = (V,E)$ is called a \textit{generalized leaf power of order $q$} if there exists a positively edge-weighted tree $(T, w)$ and an increasing sequence of thresholds $\theta = (\theta_1, \dots, \theta_q)$ such that:

\begin{itemize} 
    \item The vertex set $V$ is the set of the leaves of $T$.
    \item For any pair of distinct vertices $u, v \in V$, there is an edge $(u,v) \in E$ if and only if $d_T(u, v) \leq \theta_i$ for an odd number of thresholds $\theta_i$, where $d_T$ is the distance induced by the weight function $w$. 
\end{itemize}

Such a weighted tree is called a \emph{generalized leaf root of order $q$}.
We denote the set of generalized leaf powers of order $q$ by $GLP(q)$.

\begin{remark}
    All comparisons of distances require a finite number of bits of precision. Therefore, by rescaling, the weights and thresholds in the definition can be assumed to take integer values without loss of generality.    Consequently, the classes of Generalized leaf powers correspond to previously studied classes of graphs:

    \begin{itemize} 
        \item $GLP(1)$ is precisely the class of leaf powers.
        \item $GLP(2)$ coincides with the class of pairwise compatibility graphs (PCGs).
        \item For all $i \in \mathbb{N}$, $GLP(2i)$ corresponds to the class of $i$-interval PCGs.
    \end{itemize}
    
    The class $GLP(q)$ diverges from previously existing graph classes only for odd values of $q$. However, these classes will naturally emerge in some of the following proofs.

    Also notice that the weights on tree edges are mostly for convenience, since an edge of weight $c$ can be replaced with a subpath with $c$ edges, which preserves all distances.  Although the resulting tree may be of exponential size, requiring unweighted trees would not change our results on NP membership, see remark at the end of the next subsection.
\end{remark}
\subsection{Basic Properties}

We will now show that the $GLP(q)$ recognition problem is in NP for all $q$. It is not too difficult, though not entirely trivial, to show that LP-recognition is in NP: one can use an unweighted tree as a certificate, and verify it using a linear program whose variables are the edge lengths of that tree, and whose constraints enforce that the sum of edge weights on each leaf-to-leaf path is above or below a threshold (the latter also being a variable).  This does not work for $GLP(q)$ when $q \geq 2$, since distances can now belong to multiple ranges, which cannot be represented using linear inequalities. We therefore develop an approach that generalizes to all orders of GLP.

\begin{claim}\label{cl:numbervertices}
    Let $G\in GLP(q)$ with $n$ vertices. There exists a leaf root of order $q$ of $G$ with at most $2n-1$ vertices.
\end{claim}

\begin{proof}
    Consider a leaf root of order $q$ of $G$ with a minimal number of vertices. If it has a degree 2 vertex $u$, with neighbors $v$ and $w$, then we can contract $u$ set the weight of $(v,w)$ to be $w(u,v)+w(u,w)$ and this does not change the distances between the leaves. So the minimal leaf root of order $q$ of $G$ has no degree 2 vertices. Since it has $n$ leaves, this implies it has at most $2n-1$ vertices.
\end{proof}

 The next lemma guarantees the existence of a generalized leaf root with edge weights and thresholds that can be written using a polynomial number of bits.

\begin{lemma}\label{lem:weight-bound}
    Let $G \in GLP(q)$, and let $(T, w)$ be a generalized leaf root of $G$ of order $q$ with the corresponding sequence of thresholds $\theta = (\theta_1, \dots, \theta_q)$. There exists a new weight function $w'$ on the edges of $T$, along with a new sequence of thresholds $\theta' = (\theta'_1, \dots, \theta'_q)$, such that $(T, w')$ is also a generalized leaf root of $G$ with respect to $\theta'$, such that the weights $w'$ and thresholds $\theta'$ are integers, and such that $w'$ is upper bounded by $|E_T|^{|E_T|/2}$.
\end{lemma}

\begin{proof}
We introduce a small perturbation to each edge weight in $ E_T $ (note that the initial weights $w$ and thresholds $\theta$ are not required to be integers). This allows us to assume, without loss of generality, that for any distinct subsets $ A, B \subseteq E_T $, their weights are different; that is, $ w(A) \neq w(B) $ whenever $ A \neq B $. Consequently, the weight function $ w $ induces a total order on the power set $ \mathcal{P}(E_T) $, defined by:
$$
A \leq_w B \quad \text{if and only if} \quad w(A) \leq w(B).
$$

We consider the polytope defined in $ \mathbb{R}^{|E_T|} $ with variables $ x_1, x_2, \dots, x_{|E_T|} $, subject to the following constraints:

\begin{itemize}
    \item For all subsets $ A, B \subseteq E_T $ such that $ A \leq_w B $ and $A \neq B$, we have:
   $
   \sum_{i \in B} x_i - \sum_{i \in A} x_i \geq 1.
   $
    \item$
   x_i \geq 0.
   $
\end{itemize}

The polytope is non-empty because it contains a scaled version of the weight vector $ w $. Specifically, since $ w(A) \neq w(B) $ for $ A \neq B $, the differences $ w(B) - w(A) $ are bounded away from zero. By choosing a sufficiently large scaling factor $ c > 0 $, the vector $ x_i = c \cdot w(e_i) $ satisfies all the ordering inequalities:
$$
\sum_{i \in B} x_i - \sum_{i \in A} x_i = c \cdot \left( w(B) - w(A) \right) \geq 1.
$$

Using a well-known bound on the smallest integer solution of a system $ Ax \geq b $ (see, for example, section 1 in \cite{Lenstra}), we already obtain a solution where each coordinate is bounded by $ (|E_T| +1) \cdot |E_T|^{|E_T|/2} $. However, we can derive an improved bound and provide a much simpler proof in this case, taking advantage of the fact that $ b $ is non-negative.

Within this polytope, there exists a point that is the unique solution to a system of $ |E_T|$ independent linear constraints. By selecting $|E_T|$ linearly independent inequalities from the set of constraints (including both ordering inequalities and non-negativity constraints), we can form a system:
$$
A x = b,
$$
where $ A $ is an $ |E_T| \times |E_T| $ matrix with entries in $ \{-1, 0, 1\} $, and $ b $ is a vector with entries in $ \{0, 1\} $.

Using Cramer's rule, the unique solution to this system is given by:
$$
x_i = \frac{\det(A_i)}{\det(A)},
$$
where $ A_i $ is the matrix obtained by replacing the $ i $-th column of $ A $ with the vector $ b $. Since $ A $ and $ b $ have integer entries, both $ \det(A) $ and $ \det(A_i) $ are integers.

Because $ \det(A) $ is an integer, we can rescale the solution vector without leaving the polytope. Multiplying both the numerator and denominator by $ |\det(A)| $, we obtain the integer vector $ (|\det(A_i)|)_i $. This rescaled vector satisfies all the constraints of the polytope because $b$ is non-negative. Applying Hadamard's inequality provides an upper bound for the absolute values of the determinants:
$$
|\det(A_i)| \leq \prod_{k=1}^{|E_T|} \|A_i^{(k)}\|_2 \leq |E_T|^{|E_T|/2},
$$
where $ A_i^{(k)} $ is the $ k $-th row of $ A_i $, and $ \|A_i^{(k)}\|_2 $ is its Euclidean norm. Defining the edge weight function $ w' : e \mapsto |\det(A_e)| $ induces by construction the same ordering on $ P(E_T) $ as the original weight function $ w $. In particular, this holds for all paths between any two leaves of $ T $. Consequently, there exists a new sequence of threshold $ \theta' $ such that $ (T, w') $ is a generalized leaf root of $ G $ with respect to $ \theta' $. Furthermore, by construction, $ w' $ is bounded above by $ |E_T|^{|E_T|/2} $.
\end{proof}

\begin{lemma}\label{lem:NP}
    The $GLP(q)$ recognition problem is in NP for all $q$.
\end{lemma}

\begin{proof}
By combining Claim~\ref{cl:numbervertices} and Lemma~\ref{lem:weight-bound}, if $G \in GLP(q)$, then there exists a generalized leaf root $(T, w)$ with at most $2n - 1$ vertices, where the integral weights are bounded above by $(2n)^n$. Additionally, the thresholds in $\theta$ are all less than $n \cdot (2n)^n$ because the path between two leaves uses at most $n-1$ edges in the contracted tree of Claim~\ref{cl:numbervertices}, since there are at most $n-2$ internal vertices in $T$ because no internal vertex has degree 2. Both the tree and the thresholds can be described using $\text{poly}(n)$ bits, and verifying that it is indeed a leaf root of $G$ can also be done in polynomial time.
\end{proof}

In particular for the class of leaf powers (which is GLP(1)), we get the following theorem.

\lrank*

Combined with the well-known fact that if a graph is a $k$-leaf power, it is also a $k'$-leaf power for all $k' \geq 2k$ \cite{WB09}, we obtain the following corollary.

\begin{cor}
    A graph $G$ with $n$ vertices is a leaf power if and only if it is a $(2n)^{n+1}$-leaf power.
\end{cor}

\medskip 

\noindent 
\emph{Remark.}  Note that the NP membership result of Lemma~\ref{lem:NP} assumes that the generalized leaf root has weighted edges.  If we require such leaf roots to be unweighted, they would have exponential size and could not be used as a certificate.  This would not change the result though: we could still use the weighted tree as a certificate for the existence of the unweighted tree.

\subsection[Reducing GLP(q) Recognition to GLP(q+1) Recognition]{Reducing $GLP(q)$ Recognition to $GLP(q+1)$ Recognition}

The rest of this section is dedicated to the proof of the following key lemma.

\begin{lemma}\label{lemma:hierarchy}
    If recognizing leaf powers is NP-hard, then, recognizing $GLP(q)$ is NP-hard for all $q$.
\end{lemma}

We prove this lemma by induction on $q\geq 1$. The base case is clearly true as the set of leaf powers is exactly the class $GLP(1)$. For the induction step we use the following reduction. Assume we want to know if a given graph $G$ is a generalized leaf power of order $q$. We consider the graph $G'=(G_1\cup G_2)^C$ where $G_1$ and $G_2$ are disjoint copies of $G$ (we use the $C$ superscript on a graph to denote its complement). We will show that $G\in GLP(q)$ if and only if $G'\in GLP(q+1)$. 

\begin{proposition}\label{prop:GLP-Closure}
    For any any integer $q$, if $G$ is a generalized leaf power of order $q$, then both $G$ and $G^C$ are generalized leaf powers of order $q+1$.
\end{proposition}

\begin{proof}
    Assume that $G\in GLP(q)$. Then, there exists a tree $T$ with edge weights $w$ which is its generalized leaf root of order $q$, and there exists an increasing sequence $\theta = (\theta_1, \dots, \theta_q)$ such that $(u,v)\in E(G)$ if and only if $dist(u,v)\leq \theta_i$ for an odd number of thresholds. 

    \begin{enumerate}
        \item First, we will show that if $G\in GLP(q)$ then $G\in GLP(q+1)$:

        Set $(\theta'_{i+1})=\theta_i$ if $i\leq q$ and $\theta'_1=0$. Then, clearly $\theta'$ is non-decreasing since $\theta'_1=0\leq \theta_i=\theta'_{i+1}$. Moreover, since $w$ is positive, distances are positive and therefore $d_T(u,v)\leq \theta_i$ for an odd number of thresholds if and only if $d_T(u,v)\leq \theta'_i$ for an odd number of thresholds. This implies that $T$ is a generalized leaf root of order $(q+1)$ for $G$, so $G\in GLP(q+1)$.
        
        \item Next, we will show that if $G\in GLP(q)$ then $G^C\in GLP(q+1)$:
        
        Let $M=\max_{u,v} d_T(u,v)$. Set $(\theta''_i)=\theta_i$ if $i\leq q$ and $\theta''_{q+1}=1+\max\set{\theta_q,M}$. Then, clearly $\theta''$ is increasing since $\theta''_{q+1}>\theta_q=\theta''_q$. Moreover $d_T(u,v)\leq \theta_i$ for an even number of thresholds if and only if $d_T(u,v)\leq \theta''_i$ for an odd number of thresholds, which is satisfied exactly for edges of $G^C$. This implies that $T$ is a generalized leaf root of order $(q+1)$ for $G^C$ so $G^C\in GLP(q+1)$.
    \end{enumerate}
\end{proof}

\begin{cor}\label{cor:reduction}
    If $G$ is a generalized leaf power of order $q$, then $G'$ is a generalized leaf power of order $q+1$.
\end{cor}

\begin{proof}
    It is not hard to see that if $G\in GLP(q)$, then $G_1\cup G_2$ the disjoint union of two copies of $G$ is also in $GLP(q)$ by taking two copies of the generalized 
    leaf root of order $q$ for $G$ separated by an edge whose weight is greater than maximum threshold of $\theta$. By Proposition~\ref{prop:GLP-Closure}, we know that if $G_1\cup G_2\in GLP(q)$, then $G'=(G_1\cup G_2)^C\in GLP(q+1)$, as desired. 
\end{proof}

So, we know that our reduction will work properly for $G\in GLP(q)$, it remains to prove that if $G\notin GLP(q)$ the reduction still works.

\begin{lm}\label{lem:reduction-conv}
    If $G\not\in GLP(q)$, then $G'\notin GLP(q+1)$.
\end{lm} 

\begin{proof}
    Assume that $G\notin GLP(q)$ and that $G'\in GLP(q+1)$. We remark that the non-edges of $G'$ are exactly the edges of $G_1$ and the edges of $G_2$. Moreover, those are exactly the pairs of vertices whose distances are bounded above by an even number of threshold, according to some tree $T$ and a sequence of thresholds $(\theta_1, \ldots, \theta_{q+1})$, where $\theta_{q+1}$ is the largest threshold.

    Assume first that every edge $(u, v)$ of $G_1$ satisfies $d_T(u, v) \leq \theta_{q+1}$.  Then for each such edge,  $d_T(u, v)$ is upper-bounded by an odd number of thresholds in $(\theta_1,\dots,\theta_q)$.
    Conversely, if $u$ and $v$ form a non-edge in $G_1$, their distance is upper-bounded by an odd number of thresholds in $(\theta_1, \ldots, \theta_{q+1})$ and thus by an even number of thresholds in $(\theta_1, \ldots, \theta_q)$.  It follows that by restricting $T$ to the set of leaves $V(G_1)$ and using the thresholds $\theta_1,\dots,\theta_q$, we get $G_1 \in GLP(q)$ and thus $G \in GLP(q)$, since $G_1$ is a copy of $G$.  

    We may therefore assume that $G_1$ has an edge $(u_1, v_1)$ such that $d_T(u_1, v_1) > \theta_{q+1}$.  Because $G_2$ is also a copy of $G$, we may assume for the same reasons that $G_2$ has an edge $(u_2, v_2)$ with $d_T(u_2, v_2) > \theta_{q+1}$.  Note, these are non-edges of $G'$, and since every edge between $G_1$ and $G_2$ is present in $G'$, the quadruple $(u_1, v_1, u_2, v_2)$ induces a 4-cycle in $G'$.  
    This also implies that $d_T(x, y) \leq \theta_{q+1}$ for each $x \in \{u_1, v_1\}$ and each $y \in \{u_2, v_2\}$.  
    This implies in turn that by restricting $T$ to these four vertices and using $\theta_{q+1}$ as a threshold, we obtain a leaf-root for this 4-cycle.  This is a contradiction, since the $4$-cycle is not a leaf power because leaf powers are chordal~\cite{Ray92}.
    
    
    

    So, if $G\notin GLP(q)$ then $G'\notin GLP(q+1)$, as desired.
\end{proof}

Combining \Cref{cor:reduction} and \Cref{lem:reduction-conv} complete the proof of \Cref{lemma:hierarchy}. Interestingly \Cref{lem:reduction-conv} has another consequence.
The smallest known example of a graph that is not an $i$-interval PCG, as given in \cite{GeneralizePCG}, has $\binom{4i+1}{2i} + 4i + 1 $ vertices. Based on Lemma~\ref{lem:reduction-conv}, we can construct by induction a much smaller example of size $2^{2i+1}$.

The smallest graph that is not a PCG has eight vertices~\cite{durocher2015graphs,7verticesPCG}.  Our result provides, as a special case, an alternate proof on the existence of such a small non-PCG.

\begin{cor}
    $\forall q\in \N$, there exists a graph of $2^{q+1}$ vertices which is not in $GLP(q)$.
\end{cor}

\begin{proof}
    We will prove this by induction on $q$.  As a base case for $q = 1$, the graph $4$-cycle is not in $GLP(1)$ since leaf powers are known to be chordal.
    For the induction step, assume that we have a graph $G$ of $2^{q+1}$ vertices with $G\notin GLP(q)$. Then, by Lemma~\ref{lem:reduction-conv}, $G' = (G_1\cup G_2)^C\notin GLP(q+1)$ where $G_1$ and $G_2$ are copies of $G$. In particular, $G'$ has $2^{q+2}$ vertices, as desired.
    %
    %
\end{proof}

\section{Recognizing Leaf Powers is NP-Hard.}

Now that we have Lemma~\ref{lemma:hierarchy}, to conclude the proof of Theorem~\ref{thm:main}, it remains to prove that recognizing leaf powers is NP-hard. In this section, we establish this result, beginning with key properties of tree metrics.

\subsection{Tree Metric Properties}\label{sec:Tree Metric}

In this section, we introduce the properties of tree metrics that will be used in the reduction. Throughout this section, let $(T,w)$ be a weighted tree on the vertex set $V$ with positive weights, and let $d_T$ denote the distance on $V$ induced by $(T,w)$. The following theorem is the classical \emph{four point condition} \cite{BUNEMAN197448}, which can be derived by examining all possible configurations of paths connecting four vertices on a tree. The four point condition fully characterizes the set of possible tree metrics, which will be crucial in our reduction.

\begin{theorem}\cite{BUNEMAN197448} \label{thm:4PC}
For all $ x,y,z,t\in V$ one of the following must be true:
    \begin{enumerate}
        \item $d_T(x,y)+d_T(t,z)=d_T(x,z)+d_T(t,y)>d_T(y,z)+d_T(t,x)$
        \item $d_T(x,y)+d_T(t,z)=d_T(y,z)+d_T(t,x)>d_T(x,z)+d_T(t,y)$
        \item $d_T(x,z)+d_T(t,y)=d_T(y,z)+d_T(t,x)>d_T(x,y)+d_T(t,z)$ 
        \item $d_T(x,y)+d_T(t,z)=d_T(x,z)+d_T(t,y)=d_T(y,z)+d_T(t,x)$.
    \end{enumerate}
\end{theorem}

From this theorem we will derive two useful lemmas. 

Informally, the first lemma states that if there exists two vertices $x$ and $y$ and two pairs of vertices such that $x$ is closer to one pair and $y$ is closer to the other pair, then both vertices of either pair will be closer to the same vertex of the other pair.

\begin{lemma}\label{lem:4-PC-split}
For all $a_1, a_2, b_1, b_2 \in V$, if there exist $x, y \in V$ such that:
$$
\max(d_T(a_1, x), d_T(a_2, x)) < \min(d_T(b_1, x), d_T(b_2, x))
$$
and
$$
\max(d_T(b_1, y), d_T(b_2, y)) < \min(d_T(a_1, y), d_T(a_2, y)),
$$
then
$$
d_T(a_1, b_1) + d_T(a_2, b_2) = d_T(a_1, b_2) + d_T(a_2, b_1).
$$
In particular,
$$
d_T(a_1, b_1) \leq d_T(a_1, b_2) \iff d_T(a_2, b_1) \leq d_T(a_2, b_2).
$$

\end{lemma}

\begin{proof}
We apply Theorem~\ref{thm:4PC} with $a, x, b, y$ (where $a$ represents either $a_1$ or $a_2$ and $b$ represents either $b_1$ or $b_2$ by symmetry). The condition on $x$ and $y$ implies that $d_T(b, x) + d_T(a, y) > d_T(a, x) + d_T(b, y)$,
so the only possible case in Theorem~\ref{thm:4PC} is:
$$
d_T(a, b) + d_T(x, y) = d_T(a, y) + d_T(b, x) > d_T(a, x) + d_T(b, y). \label{4-PC-guarantee}
$$
In particular, by taking $a = a_1$ and $b = b_1$ for the first equation, and $a = a_1$ and $b = b_2$ for the second, we obtain:
$$
d_T(x, y) + d_T(a_1, b_1) = d_T(b_1, x) + d_T(a_1, y),
$$
$$
d_T(x, y) + d_T(a_1, b_2) = d_T(b_2, x) + d_T(a_1, y).
$$
By isolating $d_T(a_1, y) - d_T(x, y)$ in both equations, we get:
$$
d_T(a_1, b_1) - d_T(b_1, x) = d_T(a_1, b_2) - d_T(b_2, x).
$$
Applying the same argument with $a = a_2$ instead of $a = a_1$, we obtain:
$$
d_T(a_2, b_1) - d_T(b_1, x) = d_T(a_2, b_2) - d_T(b_2, x).
$$
By combining both, we have:
$$
d_T(a_1, b_1) + d_T(a_2, b_2) = d_T(a_1, b_2) + d_T(a_2, b_1).
$$

\end{proof}
\begin{figure}
    \centering
    \begin{tikzpicture}
        \node (a1) at (-3,1) {$a_1$};
        \node (a2) at (-3,-1) {$a_2$};
        \node (b1) at (3,1) {$b_1$};
        \node (b2) at (3,-1) {$b_2$};
        
        \node (O1) at (-2,0) {$O_1$};
        \node (Ma) at (-0.75,0) {$M_a$};
        
        \node (Mb) at (0.75,0) {$M_b$};
        \node (O2) at (2,0) {$O_2$};

        \draw (a1) -- (O1) -- (a2);
        \draw (O1) -- (Ma) -- (Mb) -- (O2);
        \draw (b1) -- (O2) -- (b2);
    \end{tikzpicture}
    \caption{If we satisfy the conditions of Lemma~\ref{lem:4-PC-split}, we must be in this configuration. $x$ must connect to the left of $M_a$, where vertices are closest to $a_1$ and $a_2$ and $y$ connects to the right of $M_b$, where vertices are closest to $b_1$ and $b_2$.}
    
\end{figure}
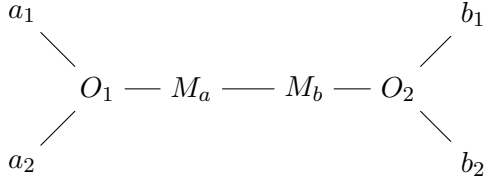

\begin{lemma}\label{lem:twins}
    Let $a_1,a_2,b,c \in V$ such that $d_T(a_1,b)< d_T(a_2,b)$ and $d_T(a_1,c)< d_T(a_2,c)$. If there exists a vertex $x$ such that $d_T(a_2,x)<d_T(a_1,x)<d_T(b,x)<d_T(c,x)$, then $d_T(a_2,b) < d_T(a_2,c)$.
\end{lemma}

Informally, the above lemma states that if $a_1$ is closer to two vertices of a pair than $a_2$, then any vertex closest to $a_2$ than to the other vertices will be closer to the same vertex of the pair as $a_2$.

\begin{proof}
  We apply Theorem~\ref{thm:4PC} with $a_1, a_2, b, x$. The conditions of the lemma imply that
$$
d_T(a_1, x) + d_T(a_2, b) > d_T(a_2, x) + d_T(a_1, b),
$$
so the only possible case in Theorem~\ref{thm:4PC} is:
$$
d_T(a_1, x) + d_T(a_2, b) = d_T(b, x) + d_T(a_1, a_2) > d_T(a_2, x) + d_T(a_1, b).
$$
Using the same argument with $c$, we get:
$$
d_T(a_1, x) + d_T(a_2, c) = d_T(c, x) + d_T(a_1, a_2) > d_T(a_2, x) + d_T(a_1, c).
$$
In particular we have the two following equations:
$$
d_T(a_1, x) + d_T(a_2, b) = d_T(b, x) + d_T(a_1, a_2),
$$
$$
d_T(a_1, x) + d_T(a_2, c) = d_T(c, x) + d_T(a_1, a_2).
$$
Subtracting the second equation from the first yields:
$$
d_T(a_2, b) - d_T(a_2, c) = d_T(b, x) - d_T(c, x).
$$
Since $d_T(b, x) < d_T(c, x)$ by assumption, the right side is negative, and therefore
$$
d_T(a_2, b) < d_T(a_2, c).
$$

\end{proof}
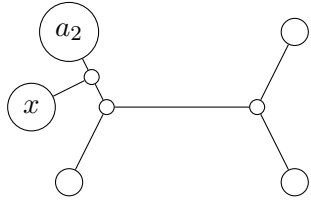
\begin{figure}
    \centering
    \begin{tikzpicture}
        \node[draw,circle ] (a2) at (-1.5,1) {$a_2$};
        \node[draw,circle ] (a) at (-1.5,-1) {};
        \node[draw,circle ] (b) at (1.5,1) {};
        \node[draw,circle ] (c) at (1.5,-1) {};
        
        \node[draw, circle, inner sep = 2pt] (O3) at (-1.2,0.4) {};
        
        \node[draw,circle] (x) at (-2,0) {$x$};
        
        \node[draw, circle, inner sep = 2pt] (O1) at (-1,0) {};
        \node[draw, circle, inner sep = 2pt] (O2) at (1,0) {};

        \draw (a2) -- (O3) -- (x);
        \draw (a) -- (O1) -- (O2);
        \draw (b) -- (O2) -- (c);
        \draw (O3) -- (O1);
    \end{tikzpicture}
    \caption{By the conditions of Lemma~\ref{lem:twins} the other 3 vertices are all closer to $a_1$ than to $a_2$, so the only place $x$ can be is on the branch connecting $a_2$ to the tree formed by the other 3 leaves.}
\end{figure}

\subsection{The Triangle Ordinal Clustering Problem}

In this section, we will describe the \emph{triangle ordinal clustering} problem which we wish to reduce it to the leaf powers recognition problem. 

The triangle ordinal clustering problem, introduced and shown to be NP-hard in \cite{SF06}, is defined as follows:

A \emph{triangle order} on a set $S$ is a strict partial order $\prec_S$ on the set of unordered pairs of elements from $S$, induced by strict total orderings within each triple of elements.  Specifically, for every triple $i,j,k\in S$, there is a strict total ordering among the three pairs $\{i,j\}, \{i,k\}, \{j,k\}$ for all $i, j, k \in S$. The triangle order $\prec_S$ is formed by combining all these total orderings across all such triples. To simplify the notation we will write $ij \prec_S ik$ instead of $\{i,j\} \prec_S \{i,k\}$.  Note that we do not assume that $i, j, k$ are distinct.  When $i = j \neq k$, we slightly abuse notation and consider $\{i, j\} = \{i, i\}$ as a pair, and assume that $i i \prec_S i k$ (otherwise we reject the instance).  Also notice that when $i = j = k$, there are no pairs to order as we consider a strict order. 

In the \emph{triangle ordinal clustering problem} (TOC), the input consists of a set $ S $ and a triangle order $ \prec_S $ on $ S $. The goal is to determine whether there exists a weighted tree $ (T, w) $ whose leaves correspond to the elements of $ S $, such that the induced distances $ d_T(i, j) $ between leaves $ i $ and $ j $ realize $ \prec_S $. This means that for every triple $ i, j, k \in S $, the ordering of the distances $ d_T(i, j), d_T(i, k), d_T(j, k) $ matches the specified ordering of the pairs $ \{i, j\}, \{i, k\}, \{j, k\} $ in $ \prec_S $.
We then say that $(T, w)$ \emph{realizes} $\prec_S$.
\subsection{Intuition of the Reduction}

In this section, we will provide the intuition behind our reduction from the triangle ordinal clustering problem to the leaf powers recognition problem.

In both the leaf powers recognition problem and the triangle ordinal clustering problem, we wish to verify the existence of a tree that satisfies various distance properties over a set of vertices. So, it is natural to consider a reduction in which the triangle order $\prec_S$ is realized by a tree $T$ which is the leaf root of some graph $G_S$.

However, for leaf powers, we can only keep the information of whether the vertices are close or far, so having a graph $G_S$ over $S$ is not enough to determine the order of the distances. So, we need to add vertices of the form $u_{i,j}$ such that $u_{i,j}$ is a witness of the set of vertices of $S$ that are closer to $i$ than $j$. If we can guarantee that $i$ and its corresponding set of vertices $U_i$ all have the same ordering over $S\setminus i$, we can use the property that each $u_{i,j}$ separates the vertices closer to $i$ than $i$ from those further away from $i$ than $j$ to get a total ordering over $S\setminus i$ for $i$ which is consistent with $\prec_S$. By having vertices close to $i$ and $U_i$ but not to the rest of the graph, we can ensure that those vertices are consistent with one another.

Consider $G'_S$ which has vertex set $S\cup \left(\bigcup_{(i,j)\in S^2} \{u_{i,j}\}\right)$. If we only needed to guarantee that $G'_S$ is a leaf power when $\prec_S$ is realizable, it is possible to construct a leaf root of $G'_S$ from the tree realizing $\prec_S$. However, due to various technicalities, the reverse direction is not so straightforward. That is, we do not have enough structure on $G'_S$ to guarantee that a leaf power on $G'_S$ corresponds to a tree realizing $\prec_S$. 

The four point condition, or to be more precise Lemma~\ref{lem:4-PC-split}, can help us get an ordering on sums of pairs of distances which can be used to guarantee most orderings but it is not enough to prove that $\prec_S$ is realized. 

First, in order to make sure that all vertices are on the same side of $i$ with respect to $U_i$, we add the vertex $O$ which is a witness that the vertices of $S$ are all ``close'' giving us extra structure. Next, instead of only considering $S$, we consider $S'$ in which each $i\in S$ has a copy $i^+$ and $i^-$ so that one of $i^+$ or $i^-$ must satisfy all the total orderings over $S'$. For our reduction, we will show that the $i^-$ copies of the vertices will respect the triangle order $\prec_S$ as needed.

This motivates the definition of $G_S$ in Section~\ref{sec:reduction}. Next, in Section~\ref{sec:TOC-to-LP}, we formally prove that if $\prec_S$ is realizable then $G_S$ is a leaf power. Finally, in Section~\ref{sec:LP-to-TOC} we formally prove that if $G_S$ is a leaf power then $\prec_S$ is realizable.

\subsection{The reduction}\label{sec:reduction}

In this section, we now proceed to describe our reduction formally. 

Given a set $S$ and a triangle order $\prec_S$ over the pairwise distances, we create a new set $S'$ by introducing two copies $i^+,i^-$ for every element $i$ of $S$. We extend the triangle order of $S$ to a new triangle order $\prec_{S'}$. For all $i\in S$, if $ij\prec_S ik$, then $\forall (i',j',k')\subseteq \set{i^+,i^-}\times \set{j^+,j^-}\times \set{k^+,k^-}$ we have $i'j'\prec_{S'} i'k'$\footnote{We do not assume that $i$ and $j$ are distinct, in particular we get that $i^+i^-\prec_{S'} i^+ k$ and $i^- i^+ \prec_{S'} i^- k$}. 
For all $x\in S'$, for all $i\in S$ with $x\notin \{i^+,i^-\}$, we have $xi^+\prec_{S'} xi^-$. Finally, $\forall i\in S$ we have $i^+i^+ \prec_{S'} i^+i^-$ and $i^-i^- \prec_{S'} i^-i^+$. We need to replace $S$ by $S'$ in order to add more structure to the vertices in order to apply the lemmas from Section~\ref{sec:Tree Metric}. The triangle order $\prec_{S'}$ is given

\begin{table}[ht!]
    \centering
    \begin{tabular}{c|c}
         Case & $\prec_{S'}$\\
         \hline
         $x\neq y$ & $xx\prec_{S'} xy$\\
         \hline
         $\set{x,y}=\set{i^+,i^-}$ and $z\notin \set{i^+,i^-}$ & $xy\prec_{S'} xz$\\
         \hline
         $ij\prec_S ik$, $(x,y,z)=(i^\pm,j^\pm,k^\pm)$ & $i^\pm j^\pm\prec_{S'} i^\pm k^\pm$\\
    \end{tabular}
    \caption{The triangle order $\prec_{S'}$ given $\prec_S$}
    \label{tab:S'}
\end{table}

We define the graph $G_S = (V_S, E_S)$, where the vertex set $V_S$ consists of the following:

\begin{itemize}
    \item For each $x \in S'$, a vertex $v_x \in V_S$.
    \item For each $x, y \in S'$, a vertex $u_{x,y} \in V_S$\footnote{We remark that $u_{x,y}$ and $u_{y,x}$ are distinct vertices and that $u_{x,x}$ is also a vertex}. Denote $U_x=\set{u_{x,y}: y\in S'}$. These vertices can be seen as witnesses to the set of vertices closer to $x$ than $x$ is to $y$.
    \item A special vertex $O \in V_S$. This vertex is a witness of the proximity of the vertices $v_x$ for $x\in S'$.
\end{itemize}

The edge set $E_S$ is defined as follows:

\begin{itemize}
    \item For each $x \in S'$, an edge $(O, v_x) \in E_S$. 
    \item For each $x, y \in S'$, an edge $(v_x, v_y) \in E_S$. 
    \item For each $x, y, z \in S'$, an edge $(u_{x,y}, u_{x,z}) \in E_S$. That is, $U_x$ is a clique. This will ensure that for fixed $x$, all vertices of $U_x$ are close and will be consistent with one another with respect to the vertices $v_y$ for $y\in S'$.
    \item For each $x, y, z \in S'$, if $xy \prec_{S'} xz$, an edge $(u_{x,z}, v_y) \in E_S$. Combined with the consistency obtained from the previous type of edges, this will ensure that $U_x$ has the order we desire with respect to $v_y$ for $y\in S'$.
\end{itemize}

\begin{figure}
    \centering
    \begin{tikzpicture}
        \node[draw,circle] (i) at (0,0) {$v_{i^\pm}$};
        
        \node[draw,circle] (O) at (0:3) {$O$};
        \node[draw,circle] (y) at (45:3) {$v_{y}$};
        \node[draw,circle] (x) at (90:3) {$v_{x}$};
        \node[draw,circle] (i') at (135:3) {$v_{i^\mp}$};
        
        \node[draw,circle] (i y) at (180:3) {$u_{i^\pm,y}$};
        \node[draw,circle] (i x) at (225:3) {$u_{i^\pm,x}$};
        \node[draw,circle] (i i') at (270:3) {$u_{i^\pm,i^\mp}$};
        
        \node[draw,circle] (i i) at (315:3) {$u_{i^\pm,i^\pm}$};

        \draw (O) -- (x) -- (y) -- (O);
        \draw (O) -- (i) -- (i') -- (O);
        \draw (x) -- (i) -- (y) --  (i') -- (x);

        \draw (x) -- (i y) -- (i') -- (i x);

        \draw (i x) -- (i y) -- (i i') -- (i i) -- (i x) -- (i) -- (i y) -- (i i);
        \draw (i) -- (i i') -- (i x);
    \end{tikzpicture}
    \caption{Subgraph generated by $i^\pm\in S'$ with $\set{x,y}\subseteq S'\setminus \set{i^{\pm}}$ such that $i^\pm x\prec_{S'} i^\pm y$.}
    \label{fig:Our Gadget}
\end{figure}
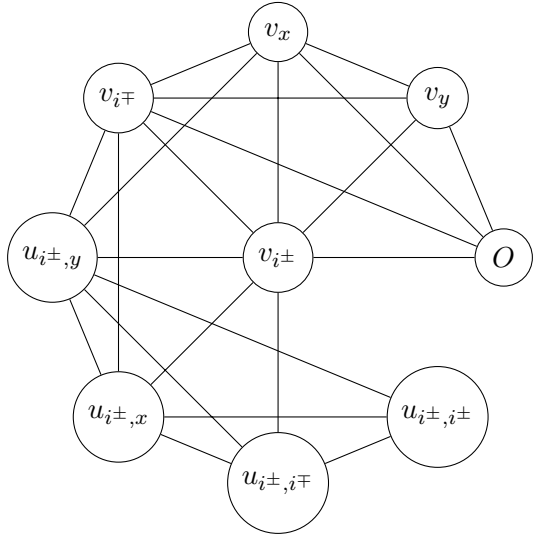

We summarize the information on the graph in Table~\ref{tab:G_S}.

\begin{table}[ht!]
    \centering
    \begin{tabular}{c|c|c}
         $u$ & $\Gamma(u)\cap V_S$ & $\Gamma(u)\cap U_x$\\
         \hline
         $v_y$ & $V_S$ & $\set{u_{x,z}:z\in S'\text{ and }xy\prec_{S'} xz}
         $\\
         \hline
         $u_{y,z}$ & $\set{v_t:t\in S'\text{ and }yt\prec_{S'}yz}$ & $\begin{cases}
             U_x\text{ if }x=y\\
             \emptyset\text{ otherwise}
         \end{cases}$\\
         \hline
         $O$ & $V_S$ & $\emptyset$
    \end{tabular}
    \caption{The neighborhoods of the vertices in $G_S$}
    \label{tab:G_S}
\end{table}

This concludes the construction.  We now argue that this produces an equivalent instance in the next sections.

\subsection{Realizable Triangle Orders Correspond to Leaf Powers}\label{sec:TOC-to-LP}

In this section, we will prove that any realizable triangle order gets mapped to a graph which is a leaf power.

\begin{lemma}\label{lem:TOC to LP}
    If $\prec_S$ is realizable by some weighted tree, then $G_S$ is a leaf power.
\end{lemma}

\begin{proof}
Assume that $\prec_S$ is realizable by some weighted tree $T'$ of diameter at least $6$\footnote{This is without loss of generality as if $T'$ has a smaller diameter, the tree obtained by multiplying all edge weights by 6 still realizes $\prec_S$}.  We argue that the existence of $T'$ implies that there exists a tree $T$ which is a leaf root of $G_S$. We can obtain $T$ from $T'$ by doing the following sequence of operations:

\begin{enumerate}
    \item Multiply the weight of each edge of $T'$ by 4. Then, relabel each leaf $i\in S$ to be $p'_i$. \label{tree:mult}
    \item Add an edge of weight $5\diam(T')$ from any non-leaf vertex of the tree to a vertex $O$.\label{tree:O}
    \item For each $i\in S$  add an edge of weight 1 from $p'_i$ to $p_{i^+}$ and an edge of weight 2 from $p'_i$ to $p_{i^-}$. \label{tree:duplicate}
    \item For each $x\in S'$, we add an edge of weight 1 from $p_x$ to $v_x$. Then, add an edge of weight $5\diam(T')$ from $p_x$ to a new vertex that we call $O_{x}$.\label{tree:S'}
    \item For each $x,y\in S'$, add an edge of weight $5\diam(T')-d_T(p_x,v_y)$ from $O_x$ to $u_{x,y}$. Namely, the subtree induced by $O_x$ and $U_{x}$ is a star.\label{tree:u}
\end{enumerate}

We claim that the tree $T$ obtained from $T'$ by these steps is a leaf root of $G_S$ with threshold $10\diam(T')-1$. To prove this, we need to check the distances which are then summarized in Table~\ref{tab:dist}) are correct. 

\begin{table}[ht!]
    \centering
    \begin{tabular}{c|c|c}
         $u$ & $v$ & $d_T(u,v)$\\
         \hline
         $p'_i$ & $p'_j$ & $4d_{T'}(i,j)$\\
         \hline
         $O$ & $\argmin_{w\in V(T')} \set{d_T(O,w)}$ & $5\diam(T')$\\
         \hline
         $p'_i$ & $p_{i^+}$ & $1$\\
         \hline
         $p'_i$ & $p_{i^-}$ & 2\\
         \hline
         $p_x$ & $O_x$ & $5\diam(T')$\\
         \hline
         $u_{x,y}$ & $O_x$ & $5\diam(T')-d_T(p_x,v_y)$
    \end{tabular}
    \caption{Distances in $T$}
    \label{tab:dist}
\end{table}

We illustrate in Figure~\ref{fig:Leaf Root} below how $T$ is constructed.

\begin{figure}[ht!]
    \centering
    
    \begin{tikzpicture}
        \node[draw,circle] (P'i) at (0,0) {$p'_{i}$};
        \node[draw,circle] (Px) at (0,-2) {$v_{x}$};
        \node[draw,fill,circle,inner sep=0pt,minimum size=0.1cm] (pO) at (0,-1) {};
        \node[draw,circle] (Pi+) at (1.5,1.5) {$p_{i^+}$};
        \node[draw,circle] (Pi-) at (-1.5,1.5) {$p_{i^-}$};

        \node[draw,circle] (i+) at (1.5,3.5) {$v_{i^+}$};
        \node[draw,circle] (i-) at (-1.5,3.5) {$v_{i^-}$};

        \node[draw,circle] (Oi+) at (4,1.5) {$O_{i^+}$};
        \node[draw,circle] (Oi-) at (-4,1.5) {$O_{i^-}$};

        \node[draw,circle] (i+x) at (4,3.5) {$u_{i^+,x}$};
        \node[draw,circle] (i-x) at (-4,3.5) {$u_{i^-,x}$};
        
        \node[draw,circle] (O) at (-2,-1) {O};
        
        \draw (P'i) to node [right] {$d$} (Px);

        \draw (i-) to node [right] {1} (Pi-);
        \draw (i+) to node [left] {1} (Pi+);

        \draw (Pi-) to node [above] {2} (P'i);
        \draw (Pi+) to node [above] {1} (P'i);

        \draw (Pi-) to node [below] {$5D$} (Oi-);
        \draw (Pi+) to node [below] {$5D$} (Oi+);

        \draw (Oi+) to node[right] {$5D-(d+1)$} (i+x);
        \draw (Oi-) to node[left] {$5D-(d+2)$} (i-x);

        \draw[dashed] (pO) to node[above] {$\geq 5D$} (O);
    \end{tikzpicture}
    \caption{Subgraph generated by $i\in S$ and $x\in S'$. We denote $\diam(T')=D$ to simplify the notation of the edge weights. Additionally, the distance between $v_x$ and $p'_i$ in $T$ to be $d$.}
    \label{fig:Leaf Root}
\end{figure}
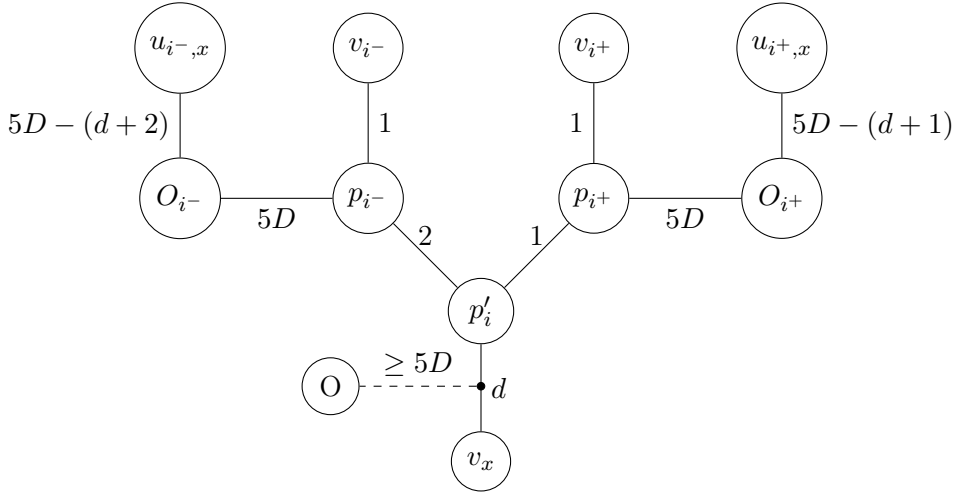

\begin{claim}\label{cl:p-dist-exact}
    $\forall x,y\in S'$ we have the following:
    \begin{enumerate}
        \item $d_T(p_x,v_y)=1$ if and only if $x=y$.
        \item $d_T(p_x,v_y)=4$ if and only if $\set{x,y}=\set{i^+,i^-}$ for some $i\in S$.
        \item If $x\in \set{i^+,i^-}$ and $y\in \set{j^+,j^-}$ for some $i\neq j$, then $d_T(p_x,v_y)>4$. Additionally:
        \begin{enumerate}
            \item If $x=i^+$ and $y=j^+$, then $d_T(p_x,v_y)=4d_{T'}(i,j)+3$.
            \item If $x=i^-$ and $y=j^+$, then $d_T(p_x,v_y)=4d_{T'}(i,j)+4$.
            \item If $x=i^+$ and $y=j^-$, then $d_T(p_x,v_y)=4d_{T'}(i,j)+4$.
            \item If $x=i^-$ and $y=j^-$, then $d_T(p_x,v_y)=4d_{T'}(i,j)+5$.
        \end{enumerate}
    \end{enumerate}
\end{claim}

\begin{proof}
    We will prove each case individually:
    \begin{enumerate}
        \item If $x=y$, then by construction of the the edge $(p_x,v_x)$ in step~\ref{tree:S'} we have $d_{T}(p_x,v_y) = 1$.  Moreover, if $x \neq y$, then one can check that no edge is added between $p_x$ and $v_y$, and the path between the two must traverse at least two edges, so that $d_T(p_x, v_y) > 1$. This can be seen in Figure~\ref{fig:Leaf Root}.
        \item If $\set{x,y}=\set{i^+,i^-}$ for some $i\in S$, then by the construction the edges $(v_{i^+},p_{i^+})$ and $(p_{i^-},v_{i^-})$ in step~\ref{tree:S'} and the construction of $(p_{i^+},p'_{i})$ in step~\ref{tree:duplicate}, these 3 edges all have weight 1. Moreover, by the construction of the edge $(p'_{i},p_{i^-})$ in step~\ref{tree:duplicate}, it has weight 2. So:
        \begin{itemize}
            \item If $x=i^+$ and $y=i^-$, then:
            \[d_{T}(p_x,v_y)=d_{T}(v_{i^+},p_{i^+})+d_{T}(p_{i^+},p'_{i})+d_{T}(p'_{i},p_{i^-})=1+1+2=4.\]
            \item If $x=i^-$ and $y=i^+$, then:
            \[d_{T}(p_x,v_y)=d_{T}(v_{i^-},p_{i^-})+d_{T}(p_{i^-},p'_{i})+d_{T}(p'_{i},p_{i^+})=1+2+1=4.\]
        \end{itemize}
        We get $d_{T}(p_x,v_y)=4$ in either case. Figure~\ref{fig:Leaf Root}.
        \item $\forall i,j\in S$ with $i \neq j$, by construction of the edges $(p'_i,p_{i^+})$, $(p'_i,p_{i^-})$, $(p'_j,p_{j^+})$ and $(p'_j,p_{j^-})$ in step~\ref{tree:duplicate} and the construction of $(p_{i^+},v_{i^+})$, $(p_{i^-},v_{i^-})$, $(p_{j^+},v_{j^+})$ and $(p_{j^-},v_{j^-})$ in step~\ref{tree:S'}, we have $2=d_{T}(v_{i^+},p'_{i})=d_{T}(v_{j^+},p'_{j})$ and $3=d_{T}(v_{i^-},p'_{i})=d_{T}(v_{j^-},p'_{j})$. Additionally, we get $d_{T}(p'_{i},p'_{j})=4 d_{T'}(i,j)$ by the construction of $T$ in step~\ref{tree:mult}. This gives the following:
        \begin{enumerate}
            \item $d_{T}(p_{i^+},v_{j^+})=d_{T}(p_{i^+},p'_{i})+d_{T}(p'_{i},p'_{j})+d_{T}(p'_{j},v_{j^+})=1+4d_{T'}(i,j)+2 =4d_{T'}(i,j)+3 >4.$
            \item $d_{T}(p_{i^+},v_{j^-})=d_{T}(p_{i^+},p'_{i})+d_{T}(p'_{i},p'_{j})+d_{T}(p'_{j},v_{j^-})=1+4d_{T'}(i,j)+3 =4d_{T'}(i,j)+4 >4.$
            \item $d_{T}(p_{i^-},v_{j^+})=d_{T}(p_{i^-},p'_{i})+d_{T}(p'_{i},p'_{j})+d_{T}(p'_{j},v_{j^+})=2+4d_{T'}(i,j)+2 =4d_{T'}(i,j)+4 >4.$
            \item $d_{T}(p_{i^-},v_{j^-})=d_{T}(p_{i^-},p'_{i})+d_{T}(p'_{i},p'_{j})+d_{T}(p'_{j},v_{j^-})=2+4d_{T'}(i,j)+3 =4d_{T'}(i,j)+5 >4.$
        \end{enumerate}
    \end{enumerate}
\end{proof}

\begin{claim}\label{cl:p-dist}
    $\forall x,y,z\in S'$ with $y \neq z$, we have $xy\prec_{S'} xz$ if and only if:
    $$d_{T}(p_x,v_y)< d_{T}(p_x,v_z)\leq 4\diam(T')+5.$$
\end{claim}
\begin{proof}
    For all $i,j$, we remark that $d_{T'}(i,j)\leq \diam(T')$ by the definition of the diameter. We can use this alongside Claim~\ref{cl:p-dist-exact} to deduce $d_{T}(p_x,v_z)\leq 4\diam(T')+5$ in all cases:
    \begin{enumerate}
        \item If $x=z$ we get $d_{T}(p_x,v_z)=1\leq 4\diam(T')+5$. 
        \item If $\set{x,z}=\set{i^+,i^-}$, then $d_T(p_x,v_z)=4\leq 4\diam(T')+5$.
        \item If $x\in\set{i^+,i^-}$ and $z\in \set{j^+,j^-}$ for $i\neq j$, then $d_T(p_x,v_z)\leq 4d_{T'}(i,j)+5\leq 4\diam(T')+5$.
    \end{enumerate}
    
    So, the second inequality holds in all cases. It remains to show that $d_{T}(p_x,v_y)< d_{T}(p_x,v_z)$.

    \begin{itemize}
        \item Assume that $x=y\in \set{i^+,i^-}$ for some $i$. Then, by Claim~\ref{cl:p-dist-exact}, we get $d_T(p_x,v_y)=1$. If $\exists j\neq i$ such that $z\in \set{j^+,j^-}$, then we get $d_{T}(p_x,v_z)>4$ by Claim~\ref{cl:p-dist-exact}. Otherwise, since $y\neq z$ we must have $\set{x,z}=\set{i^+,i^-}$. So, by Claim~\ref{cl:p-dist-exact}, we get $d_{T}(p_x,v_z)=4$. In either case, we get $d_T(p_x,v_y)=1<4\leq d_T(p_x,v_z)$.
        \item Assume that $\set{x,y}=\set{i^+,i^-}$ for some $i\in S$, then by Claim~\ref{cl:p-dist-exact} we get $d_T(p_x,v_y)=4$. Since $xy\prec_{S'} xz$, there must exist $j\neq i$ such that $z\in \set{j^+,j^-}$. This implies that $d_T(p_x,v_z)>4$. So we get $d_{T}(p_x,v_y)=4<d_{T}(p_x,v_z)$.  
        \item Otherwise, we must have $x\in \set{i^+,i^-}$ and $y\in \set{j^+,j^-}$ for some $i\neq j$. 
        \begin{itemize}
            \item If $\{y, z\} = \{j^+, j^-\}$, then as $xj^+ \prec_{S'} xj^-$ we must have $y = j^+, z = j^-$.
            \begin{itemize}
                \item If $x=i^+$, then $d_T(p_x,v_y)=4d_{T'}(i,j)+3<4d_{T'}(i,j)+4=d_T(p_x,v_z)$.
                \item Otherwise, $x=i^-$ and $d_T(p_x,v_y)=4d_{T'}(i,j)+4<4d_{T'}(i,j)+5=d_T(p_x,v_z)$.
            \end{itemize}
            \item Otherwise, we must have $z\in \set{k^+,k^-}$ for $k\notin \set{i,j}$. So, since $xy\prec_{S'} xz$ we must have $d_{T'}(i,j)< d_{T'}(i,k)$, and thus $d_{T'}(i, j) \leq d_{T'}(i, k) - 1$ since distances are integral. 
            This implies that
            $$d_{T}(p_x,v_y)\leq 4d_{T'}(i,j)+5<4d_{T'}(i,k)+3 = d_T(p'_i, p'_k) + 3 \leq d_{T}(p_x,v_z).$$
        \end{itemize}
    \end{itemize}

    The inequalities $d_{T}(p_x,v_y)< d_{T}(p_x,v_z)\leq 4\diam(T')+5$ therefore hold.
    
    Since $\prec_{S'}$ is a triangle order, exactly one of $xy\prec_{S'} xz$ or $xz\prec_{S'} xy$ is true so we get both directions.  
\end{proof}

    We next argue that our construction of $T$ is valid, i.e., that its edge weights only consist of positive integral weights.  It is clear that the first 4 steps of the construction of $T$ only use positive integral weight edges. It remains to prove that the edge from $O_x$ to $u_{x,y}$ is also positive for all $x, y \in S'$. This follows from Claim~\ref{cl:p-dist} as if $d_T(p_x,v_y)\leq 4\diam(T')+5< 5\diam(T')$ then $5\diam(T')-d_T(p_x,v_y)>0$ as long as $T'$ has diameter at least $6$.

    We next argue that distinct vertices $u, v$ of $G_S$ share an edge if and only if $d_T(u, v) \leq 10 \diam(T') - 1$.  First observe that 
    by construction, $\forall x,y\in S'$ the path from $u_{x,y}$ to $p_x$ uses the vertex $O_x$. So: 

    $$d_{T}(p_x,u_{x,y})=d_{T}(p_x,O_x)+d_{T}(O_x,u_{x,y})>5\diam(T').$$ 

    We then get:
    
    \begin{itemize}
        \item $O$ is connected to $T$ by an edge of weight $5\diam(T')$ by step~\ref{tree:O}. Let $p_O$ be the neighbour of $O$ we get $d_T(O,p_O)=5\diam(T')$. Then, by step~\ref{tree:mult} of the construction, assuming that $x\in \set{i^+,i^-}$ for some $i\in S$, we get $d_T(p_O,p'_i)\leq 4\diam(T')$. Finally, by step~\ref{tree:duplicate}, $d_T(p'_i,p_x)\leq 2$. So, by combining these we get: 
        \begin{align*}
            \forall x\in S' \ \  &d_{T}(O,p_x)=d_T(O,p_O)+d_T(p_O,p'_i)+d_T(p'_i,p_x)\\
            \Longrightarrow\ \ & 5\diam(T')\leq d_T(O,p_x)\leq 9\diam(T')+2.
        \end{align*}
        This implies $O$ is at distance at most $9\diam(T')+3<10\diam(T')$ from all $v_x$ for $x\in S'$ and at distance at least $10\diam(T')$ from $u_{x,y}$ for all $x,y\in S'$.  The distance relationships that involve $O$ are therefore correct.
        \item By Claim~\ref{cl:p-dist}, $\forall x,y\in S'$, $d_{T}(v_x,v_y)\leq 4\diam(T')+5<10\diam(T')$, which is correct.
        \item For all $x, y, z \in S'$, consider $u_{x, y} \in U_x$ and $v_z$.  Note that all paths between $u_{x,y}$ and vertices outside of $U_x$  use $p_x$. 
        \begin{itemize}
            \item If $u_{x, y}, v_z$ share an edge, then $xz\prec_{S'}xy$.  Using Claim~\ref{cl:p-dist}, 
            $$
            d_T(u_{x,y}, v_z) = d_T(u_{x,y}, p_x) + d_T(p_x, v_z) < d_T(u_{x,y},p_x) + d_T(p_x,v_y) = 10 \diam(T')
            $$
            and so $d_T(u_{x,y}, v_z)$ is at most $10 \diam(T') - 1$ since distances are integral. 
            \item If $u_{x, y}, v_z$ is not an edge, then it must be that either $y=z$ or $xy\prec_{S'}xz$. Again using Claim~\ref{cl:p-dist}, 
            $$
            d_T(u_{x,y}, v_z) = d_T(u_{x,y}, p_x) + d_T(p_x, v_z) \geq d_T(u_{x,y},p_x) + d_T(p_x,v_y) = 10 \diam(T').
            $$
        \end{itemize}

        \item $\forall x,y,z\in S'$, then the path from $u_{x,y}$ to $u_{x,z}$ is the path from $u_{x,y}$ to $O_x$ and the path from $O_x$ to $u_{x,z}$. So, their distance is: $$d_T(u_{x,y},O_x)+d_T(O_x,u_{x,z})=5\diam(T')-d_T(p_x,v_y)+5\diam(T')-d_T(p_x,v_z)<10\diam(T').$$
        
        So $U_x$ is a clique.
        \item $\forall x_1,y_1,x_2,y_2$ with $x_1\neq x_2$, the distance from $u_{x_1,y_1}$ to $u_{x_2,y_2}$ is at least the distance from $O_{x_1}$ to $p_{x_1}$ and the distance from $O_{x_2}$ to $p_{x_2}$. Since both distances are $5\diam(T)$, this implies that $d_T(u_{x_1,y_1},u_{x_2,y_2})\geq 10\diam(T')$.
    \end{itemize}
    This defines every case so $T$ is a leaf root of $G_S$.
\end{proof}

It remains to prove that if $G_S$ is a leaf power, then $\prec_S$ is realizable.

\subsection{Leaf Powers Correspond to Realizable Triangle Orders}\label{sec:LP-to-TOC}

In this section, we will prove that any triangle order which is mapped to a leaf power must be realizable.

In the remainder, we let $T$ be a leaf root of $G_S$.
We show that the subgraph induced by the tree containing the leaves of the form $v_{i^-}$ for $i\in S$ realizes $\prec_S$. That is, we wish to show that $\forall i,j,k\in S$, if $ij\prec_S ik$ then $d_T(v_{i^-},v_{j^-})<d_T(v_{i^-},v_{k^-})$.

First, we need to prove two claims. 

\begin{claim}
    \label{claim 1}
    $\forall i\in S$, $\forall z\in S' \setminus \{i^+, i^-\}$:
    $$d_T(v_{i^+}, v_z)<d_T(v_{i^-}, v_z)$$
\end{claim}

That is, for any vertex corresponding to $S'$ except $v_{i^+}$ and $v_{i^-}$ will be closer to $v_{i^+}$ than to $v_{i^-}$

\begin{proof}
    We apply Lemma~\ref{lem:4-PC-split} for $a_1=v_{i^+}$, $a_2=v_{i^-}$, $b_1=v_{z}$, $b_2=u_{z,i^-}$ with $x=u_{i^+,z}$ and $y=u_{z,i^+}$. 
    Since $ii \prec_S ij$ for $j \neq i$, by construction we have $i^+ i^+ \prec_{S'} i^+ z$ and $i^+ i^- \prec_{S'} i^+ z$, implying that $(u_{i^+, z}, v_{i^+})$ and $(u_{i^+, z}, v_{i^-})$ are edges of $G_S$.  Moreover, $i^+ z \not \prec_{S'} i^+z$ and so $(u_{i^+,z}, v_z)$ is \emph{not} an edge, and $(u_{i^+,z},u_{z,i^-})$ is not an edge since $z \neq i^+$.  Since $T$ is a leaf root of $G_S$, it follows that $\max(d_T(v_{i^+}, u_{i^+,z}), d_T(v_{i^-}, u_{i^+,z})) < \min(d_T(v_z, u_{i^+,z}), d_T(u_{z,i^-}, u_{i^+,z}))$, which satisfies the first condition of Lemma~\ref{lem:4-PC-split}.  The second condition of the lemma can easily be verified in a similar manner (the only difference in the argument being that $(u_{z,i^-}, u_{z,i^+})$ is an edge because $U_z$ is a clique).
    We get:
    $$d_T(v_{i^+},v_{z})+d_T(v_{i^-},u_{z,i^-})=d_T(v_{i^+},u_{z,i^-})+d_T(v_{i^-},v_{z})$$
    $$\Longrightarrow d_T(v_{i^+},v_{z})-d_T(v_{i^-},v_{z})=d_T(v_{i^+},u_{z,i^-})-d_T(v_{i^-},u_{z,i^-}).$$

    Since $zi^+ \prec_{S'} zi^-$, $(u_{z,i^-},v_{i^+})$ is an edge, and because $(u_{z,i^-},v_{i^-})$ is not an edge, 
    $u_{z,i^-}$ is closer to $v_{i^+}$ than to $v_{i^-}$.  Thus, the right-hand side in the above equality is strictly below $0$, and so we must have $d_T(v_{i^+}, v_z)<d_T(v_{i^-}, v_z)$.
\end{proof}

\begin{claim}
    \label{claim 2}
    $\forall i\in S$, $\forall z,z'\in S'$ if $i^-z\prec_{S'} i^-z'$, then: 
    $$d_T(u_{i^-,i^+},v_z)<d_T(u_{i^-,i^+},v_{z'}).$$
\end{claim}

That is, the order of proximity of $u_{i^-,i^+}$ to the vertices $v_z$ for $z\in S'\setminus \{i^+,i^-\}$ is the one induced by $S'$ for $i^-$ (and $i^+$ due to the definition of $\prec_{S'}$).
\begin{proof}
    We apply Lemma~\ref{lem:4-PC-split} for $a_1=u_{i^-,i^+}$, $a_2=u_{i^-,z'}$, $b_1=v_z$, $b_2=v_{z'}$ with $x=u_{i^-,i^-}$ and $y=O$, and noting that $z \neq z'$  we get:
    $$d_T(u_{i^-,i^+},v_z)+d_T(u_{i^-,z'},v_{z'})=d_T(u_{i^-,i^+},v_{z'})+d_T(u_{i^-,z'},v_z)$$
    $$\Longrightarrow d_T(u_{i^-,i^+},v_z)-d_T(u_{i^-,i^+},v_{z'})=d_T(u_{i^-,z'},v_z)-d_T(u_{i^-,z'},v_{z'}).$$

    Because $i^-z \prec_{S'} i^-z'$, $u_{i^-,z'}$ is closer to $v_z$ than to $v_{z'}$ so we must have $d_T(u_{i^-,i^+},v_z)<d_T(u_{i^-,i^+},v_{z'})$
\end{proof}

\begin{lm}\label{lem:LP to TOC}
    If $G_S$ is a leaf power, then $\prec_S$ is realizable by some tree.
\end{lm}

\begin{proof}
    We wish to show that $\forall i,j,k\in S$, if $ij\prec_S ik$ then $d_T(v_{i^-},v_{j^-})<d_T(v_{i^-},v_{k^-})$.  Note that this trivially holds if $i = j \neq k$, and that no such relation is possible if $j = k$ or $i = k$.  We may thus assume that $i, j, k$ are distinct.
    
    First, consider two applications of Claim~\ref{claim 1} with $i$ and $\ml{z} \in \set{j^-,k^-}$. Then, we get $d_T(v_{i^+},v_{j^-})< d_T(v_{i^-},v_{j^-})$ and $d_T(v_{i^+},v_{k^-})< d_T(v_{i^-},v_{k^-})$. Next, we can consider three applications of Claim~\ref{claim 2} with $i$ 
    and every $(z, z') \in \set{(i^-,i^+),(i^+,j^-),(j^-,k^-)}$, then, for $x=u_{i^-,i^+}$ we have $d_T(v_{i^-},x)<d_T(v_{i^+},x)<d_T(v_{j^-},x)<d_T(v_{k^-},x)$. In particular, we are able to apply Lemma~\ref{lem:twins} using $a_1=v_{i^+}$, $a_2=v_{i^-}$, $b=v_{j^-}$ and $c=v_{k^-}$, which implies that $d_T(v_{i^-},v_{j^-})<d_T(v_{i^-},v_{k^-})$. This concludes the proof.
\end{proof}

We have shown that: (1) if a given TOC instance $\prec_S$ is realizable by some weighted tree, then $G_S$ is a leaf power (Lemma~\ref{lem:TOC to LP}); and (2) if $G_S$ is a leaf power, then $\prec_S$ is realizable by some weighted tree (Lemma~\ref{lem:LP to TOC}).
We therefore establish that a triangle order $\prec_S$ is realizable by some tree if and only if $G_S$ is a leaf power. Moreover, the size of $G_S$ polynomial in the size of $S$ (the graph $G_S$ has $O(|S|^2)$ vertices because of the $U_x$ sets, and $O(|S|^4)$ edges because the $U_x$'s are cliques), providing a polynomial-time reduction from the TOC problem to the leaf power recognition problem.

\begin{lemma}\label{lem:LP-NPhard}
    Recognizing leaf powers is NP-hard.
\end{lemma}

Bringing everything together, Lemmas~\ref{lem:LP-NPhard}, \ref{lem:NP}, and \ref{lemma:hierarchy} directly lead to Theorem~\ref{thm:main}. Moreover, a closer look at our reduction can slightly improve Lemma~\ref{lem:LP-NPhard}.

\unary*

\begin{proof}
    It follows directly from the proof of Lemma~\ref{lem:NP} that the problem is in NP. It remains to show that it is NP-hard, using our reduction described above.

    First consider the NP-hardness reduction of the TOC problem provided by Shah and Farach-Colton in~\cite{SF06}.
    In their proof, the authors show that for the set of triangle orders $(S,\prec_{S})$ which can arise from their reduction, $(S,\prec_{S})$ is realizable if and only if there exists a weighted tree $T_{S}$ which realizes $(S,\prec_{S})$ where $|V(T_S)|$ and all edge weights of $T_{S}$ are polynomial with respect to $|S|$. In particular, $\diam(T_{S})$ is also polynomial with respect to $|S|$.
    Let $c$ be a positive constant such that for any instance $(S, \prec_S)$ obtained from the reduction of~\cite{SF06},  $|S|^c$ is an upper bound on the diameter of a weighted tree that realizes $(S, \prec_S)$, if one exists.

    Coming back to our reduction, consider the graph $G_S$ constructed from $(S, \prec_S)$.  We argued that if $(S, \prec_S)$ is realizable by some weighted tree $T_S$, then $G_S$ is a $(10 diam(T_S) - 1)$-leaf power. 
    We claim that $G_S$ is also a $(10 |S|^c - 1)$-leaf power.
    Indeed, it is known that for any $k \geq 2$, a $k$-leaf power is also a $(k + 2)$-leaf power, and thus it is a $(k + 2d)$-leaf power for any integer $d \geq 0$~\cite{brandstadt2008k}. 
    If $G_S$ is a $(10 diam(T_S) - 1)$-leaf power, it is also a $(10 |S|^c - 1)$-leaf power, since we can add $2$ to $10 diam(T_S) - 1$ as many times as needed to reach $10 |S|^c - 1$ (note that the parity works out).

    We can therefore show that unary-$k$-leaf power is NP-hard as follows.  From a TOC instance $(S, \prec_S)$, construct graph $G_S$ and put $k = 10 |S|^c - 1$. 
    If $(S, \prec_S)$ is realizable by some tree $T_S$, then from Lemma~\ref{lem:TOC to LP} we have that $G_S$ is a $(10 diam(T_S) - 1)$-leaf power and therefore a $k$-leaf power. 
    Conversely, if $(S, \prec_S)$ is not realizable, by contraposing Lemma~\ref{lem:LP to TOC} we have that $G$ is not a leaf power.  In particular, $G_S$ is not a $k$-leaf power. 
    Since $|V(G_S)|$ and $k$ are polynomial with respect to $|S|$, this implies that the $k$-leaf power recognition problem is NP-complete when $k$ is part of the input and polynomial in the size of the input graph.
    %
    %
    %
    %
\end{proof}

\section{Conclusion}

Our results solve a longstanding problem on the complexity of recognizing leaf powers. In addition, we have shown that the hardness extends to generalizations of leaf powers based on tree distances lying in given intervals, including pairwise compatibility graphs for which the algorithmic literature was very sparse.  
There are still several directions to explore regarding the complexity aspects of these graph classes.  Notably, $k$-leaf power recognition is in P for fixed $k$~\cite{Lafond2023} but NP-complete for given $k$, but it is still unknown whether the problem is fixed-parameter tractable.  That is, can $k$-leaf powers be recognized in time $f(k) poly(n)$ for some function $f$ not depending on $n$?  A positive result would provide hope in using leaf powers in practice despite our hardness result.  Likewise, the running time of Lafond's algorithm is $n^{g(k)}$ for a function $g$ that is a tower of exponential it is open whether this time complexity can be avoided.

Another direction is to investigate the question of leaf ranks.  As previously mentioned, it is open whether every leaf power is a $k$-leaf power for some $k$ bounded by $2^{O(n)}$.  Moreover, only trivially perfect graphs are known to admit a polynomial-time algorithm for the leaf rank computation, and the complexity remains open even for simple graphs classes known to be leaf powers.  This includes, for example, unit interval graphs, and possible directions are discussed in~\cite{hogemo2025mapping}.

Let us finally mention that the ideas developed here could be applied to other graph-to-tree embedding problems.  To our knowledge, the TOC problem had not been used for other hardness results, and could prove useful in the future.  In the recent~\cite{hogemo2025mapping}, several ways of mapping graphs to trees are presented, with many open complexity problems.  One interesting example concerns graphs of mim-width $1$.  The complexity of recognizing graphs of constant mim-width has been open since their introduction~\cite{jaffke2020mimI,jaffke2020mimII}, with a recent breakthrough showing hardness for graphs of mim-width 1211~\cite{bergougnoux2025mim}. Since all leaf powers have mim-width 1, it might be possible to modify our reduction to extend it to a hardness result for this graph class.

\bibliographystyle{plain}
\bibliography{references}

\end{document}